\documentclass[11pt]{article}
\usepackage{amsmath,amssymb,amsthm}

\newtheorem{thm}{\noindent Theorem}[section]
\newtheorem{lem}{\noindent Lemma}[section]
\newtheorem{cor}{\noindent Corollary}[section]
\newtheorem{prop}{\noindent Proposition}[section]

\newtheorem{rem}{\noindent Remark}[section]

\begin{document}
\title{Frequently visited sites of the inner boundary of simple random walk range}
\author{Izumi Okada}
\date{\empty}
\maketitle

\begin{abstract}
This paper considers the question: 
how many times does a simple random walk revisit  the most frequently visited site among  the  inner boundary points?
It is known that in ${\mathbb{Z}}^2$, the number of visits to the most frequently visited site among all of the points of the random walk range up to time $n$ is asymptotic to $\pi^{-1}(\log n)^2$, while in ${\mathbb{Z}}^d$ $(d\ge3)$, it is of order $\log n$. 
We prove that the corresponding number for the inner boundary is asymptotic to $\beta_d\log n$ for any $d\ge2$, where  $\beta_d$ is a certain constant having a simple probabilistic expression.
\end{abstract}

\section{Introduction }
Many works have studied properties of  the trajectory of a simple random walk. 
These properties include the growth rate of  the trajectory's range, location of  the most frequently visited site, the number of the visits to this site,  
the number of  the sites of frequent visits, and so forth.  
There remain many interesting unsolved questions concerning  these properties. 
The most frequently visited site among all the points of the range (of the walk of finite length) is called  a  favorite site 
and a site which is revisited many times (in a certain specified sense) is called a frequently visited site. 
About fifty years ago, Erd\H{o}s and Taylor \cite{dvo1} posed a problem concerning a simple random walk in ${\mathbb{Z}^{d}}$: 
how many times does the random walk revisit the favorite site (up to a specific time)?  
Open problems concerning the favorite site are raised by Erd\H{o}s and R\'ev\'esz \cite{er2}, \cite{er3} and Shi  and T\'oth \cite{shi} but remain unsolved so far. 
By Lifshits and Shi \cite{shi2} it is known that 
the favorite site of a $1$-dimensional  random walk tends to be far from the origin, 
but almost nothing is known about its location for multi-dimensional walks. 
 
In this paper, we focus on the most frequently visited site  among   the inner boundary points  of the random walk range, rather than among all of the points of the range, and propose the  question:  
how many times does a random walk revisit the most frequently visited site  among the inner boundary points?  
Here, we briefly state our result and compare it with known results for the favorite site.  
Let $M_0(n)$ be the number of visits  to the favorite site by the walk up to time $n$ and 
$M(n)$ be that of the most frequently visited site among the inner boundary points. 
In Theorem \ref{m1}, we will prove that 
for $d\ge2$
\begin{align*}
\lim_{n \to \infty} \frac{M(n)}{\log n}= \frac{1}{- \log P(T_0<T_b)} \quad \text{ a.s.}
\end{align*}
Here, $T_a$ is the first time the random walk started at the origin  hits $a$ after time $0$, and
$b$ is a neighbor site of the origin. 
To compare, a classical result of 
Erd\H{o}s and Taylor \cite{dvo1} says that  for $d\ge3$,
\begin{align*}
\lim_{n \to \infty} \frac{M_0(n)}{\log n}= \frac{1}{- \log P(T_0<\infty)} \quad \text{ a.s.},
\end{align*}
and for $d=2$,  $M_0(n)/(\log n)^2$ is bounded away from zero and infinity a.s.  (the limit exists and is identified \cite{Dembo}  as mentioned later in Section 2). 

These results illuminate the geometric structure of the random walk range as well as the nature of recurrence or transience of random walks.
We are able to infer  that the  favorite site is outside the inner boundary from some time onwards with probability one. 
This may appear intuitively clear; it seems  improbable for the favorite  point  to continue to be  an inner boundary point since it must be visited many times, but
our result further shows that there are many inner boundary points that are  visited many times, 
with amounts comparable to that of the favorite point for $d\ge 3$.  
In addition,  the growth order of $M_0(n)$ is the same for all $d\geq 2$, meaning 
the phase transition which occurs between $d=2$  and $d\geq3$ for $M_0(n)$  does not occur  for  $M(n)$.

In Theorem \ref{m2}, which is a strong claim in comparison to Theorem \ref{m1},  we will provide an explicit answer  to the question of how many  frequently visited sites  among the inner boundary points exist. 

The upper bounds for both Theorem \ref{m1} and Theorem \ref{m2} are obtained using the idea in  \cite{dvo1}. 
The Chebyshev inequality and the Borel-Cantelli lemma  are also used in the same way as in  \cite{dvo1}. 
On the other hand, $M(n)$ is not monotone, while $M_0(n)$ is monotone.  
We work with the walk and its trajectory at the times $2^k$ and find a process that is monotone and a bit  larger than  $M(n)$, but with the desired asymptotics. 

On the other hand, the idea for the proof of the lower bound is different from that for the known results. 
In \cite{Dembo}, a Brownian occupation measure was used in the proof. 
Rosen \cite{Rosen} provided another proof to the result of \cite{Dembo}, 
in which he computed a crossing number instead of the number of the frequently visited site. 
In this paper, we use the Chebyshev inequality and the Borel-Cantelli lemma as in  \cite{Rosen} but for the number of the frequently visited sites among the inner boundary points. 
In addition, as the proof of the upper bound,  we estimate a number 
slightly smaller than the number of the frequently visited site among the inner boundary points. 

We conclude this introduction by mentioning  some known results about  the inner boundary points of the random walk range that are closely related to the present subject. 
Let $L_n$ be the number of the inner boundary points up to time $n$. 
In \cite{itai}, it is noticed that the entropy of a random walk is essentially governed 
by the asymptotic of $L_n$. 
In \cite{okada},  the law of large numbers  for $L_n$ is shown and  $\lim L_n/n$ is identified 
for a simple random walk on ${\mathbb Z}^d$ for every  $d\ge1$. 
Let $J_n^{(p)}$ denote the number of $p$-multiplicity points in the inner boundary and be defined as
\begin{align*}
J_n^{(p)}=\sharp &\{S_i\in \partial R(n):\sharp\{m:0\le m \le n,S_m=S_i \}=p \},
\end{align*}
where $\partial R(n)$ is the set of the inner boundary of $\{S_0,S_1,...,S_n\}$ and $\sharp A$ is the number of elements in $A$. 
In \cite{okada}, it is also shown that for a simple random walk in $\mathbb{Z}^2$, with $p\ge1$,
\begin{align}\notag
\frac{{\pi}^2}{2} \le \lim_{n\to \infty}EL_n\times\frac{(\log n)^2}{n} &\le 2{\pi}^2,
\\ \label{iii+}
\frac{\tilde{c}^{p-1}\pi^2}{4} \le  \lim_{n\to \infty}EJ_n^{(p)}\times\frac{(\log n)^2}{n}  
&\le \tilde{c}^{p-1}\pi^2,
\end{align}
where $\tilde{c}=P(T_0<T_b)$ for any/some neighbor site $b$ of the origin. 
These  may be compared with the results for the entire range; according to \cite{Fla},  
 $\sharp\{S_i:0\le i \le n\}$ in ${\mathbb{Z}}^2$ is asymptotic to $\pi n/\log n$  and 
the asymptotic form of  the number of $p$-multiplicity points in it is  independent of $p$. 
\section{Framework and Main Results }
Let $\{S_k\}_{k=0}^{\infty}$ be a simple random walk 
on the $d$-dimensional square lattice ${\mathbb Z}^d$. 
Let $P^a$ denote the probability of the simple random walk starting at $a$;    
we simply write $P$ for $P^0$.  
Let ${\mathbb N}=\{1,2,3,...\}$ and for $n\in{\mathbb N}$, set $R(n) = \{S_0,S_1, \ldots, S_n\}$ as the random walk range up to the time $n$. 
We call $z \in \mathbb{Z}^d$ a neighbor of $a\in \mathbb{Z}^d$ if $|a-z|=1$. 
Let ${\cal N}(a)$ be the set of all neighbors of $a$ defined as
$${\cal N}(a)=\{z\in{\mathbb Z}^d : |a-z|=1 \}.$$
The inner boundary of $A \subset \mathbb{Z}^d$ is denoted by $\partial A$, that is
$$\partial A =\{x \in A : {\cal N} (x) \not\subset A  \}.$$  
We denote the number of visits to $x$ of $S_m$, $0\le m\le n$ by $K(n,x)$. 
That is, 
$$K(n,x)=\sum_{j=0}^n1_{\{S_j=x\}}.$$
Moreover, we set $$M(n):=\max_{x \in \partial R(n)}K(n,x).$$
Clearly, this is the maximal number of visits of the random walk of length $n$ to 
$\partial R(n)$, the inner boundary of $R(n)$. 
Let $T_x$ denote the first passage time to $x$: 
$T_x=\inf\{m\ge1: S_m=x\}$. 
We are now ready to state our main theorems. 
The first theorem provides us with sharp asymptotic behavior of $M(n)$.
\begin{thm}\label{m1}
For $d\ge2$
\begin{align*}
\lim_{n \to \infty} \frac{M(n)}{\log n}=\beta_d \quad \text{ a.s.},
\end{align*}
where
\begin{align*}
\beta_d=\frac{1}{- \log P(T_0<T_b)}
\end{align*}
for any $b\in {\cal N}(0)$. 
Note that $P(T_0<T_b)$ does not depend on the choice of $b\in {\cal N}(0)$, but rather 
depends only on the dimension. 
\end{thm}
Leading to the second main theorem, 
we first define $\Theta_n (\delta)$ for $n \in \mathbb{N}$ and $0<\delta<1$ as
\begin{align*}
\Theta_n(\delta):=\sharp \{ x\in \partial R(n) :
\frac{K(n,x)}{\log n}\ge \beta_d\delta \}.
\end{align*}
This is the cardinality of points in $\partial R(n)$ whose number of visits is comparable to the maximal order appearing in Theorem \ref{m1}, with a ratio greater than or equal to $\delta$. 
Our second main theorem exhibits the sharp logarithmic asymptotic behavior of $\Theta_n(\delta)$ as $n \to\infty$. 
\begin{thm}\label{m2}
For $d\ge2$ and $0<\delta<1$,
\begin{align*}
\lim_{n \to \infty} \frac{\log \Theta_n(\delta)}{\log n}=1-\delta \quad a.s.
\end{align*}
\end{thm}
We compare our main results to the corresponding results for the whole random walk range $R(n)$.  
We denote the quantity corresponding to $M(n)$ by $M_0(n)$. 
That is, $M_0(n)=\max_{x\in{\mathbb{Z}^{d}}}K(n,x)$ where 
$M_0(n)$ represents the maximal number of visits of the random walk to a single site in the entire random walk range until time  $n$. 
Erd\H{o}s and Taylor \cite{dvo1} showed
that  for $d\ge3$
\begin{align*}
\lim_{n \to \infty} \frac{M_0(n)}{\log n}= \frac{1}{-\log P(T_0<\infty)}\quad \text{ a.s.}
\end{align*}
For $d=2$, they obtained 
\begin{align*}
\frac{1}{4\pi} \le 
\liminf_{n \to \infty} \frac{M_0(n)}{(\log n)^2}\le
\limsup_{n \to \infty} \frac{M_0(n)}{(\log n)^2}\le \frac{1}{\pi} \quad \text{ a.s.},
\end{align*}
and conjectured that the limit exists and equals $1/\pi$ a.s.  
Forty years later,  Dembo et al. \cite{Dembo} verified this conjecture and also showed how many  frequently visited sites  
of order $(\log n)^2$ there are in the following sense. 
Let $d=2$. 
Then, for  $0<a<1$,  define 
$$\Theta_{0,n}=\sharp\{ x\in {\mathbb{Z}^{2}}: 
\frac{K(n,x)}{(\log n)^2}\ge \frac{a}{\pi} \}.$$ 
Then
\begin{align*}
\lim_{n\to \infty} \frac{\log \Theta_{0,n}}{\log n}=1-a \quad \text{ a.s.}
\end{align*}

In view of these results, Theorem  \ref{m1} entails the following corollary. 
\begin{cor}
For $d\ge2$, the favorite site does not appear in the inner boundary from some time onwards a.s.  
In other words, $M_0(n)>M(n)$ for all but finitely many $n$ with probability one. 
\end{cor}

The $p$-th hitting times  $T^p_x$ for $p=0,1\,\ldots$ and  the partial ranges  $R[l,n]$  that we are now to define play significant roles. 
 Let $T_x^0=\inf\{j\ge0: S_j=x\}$ and for $p\ge1$,
 \begin{align} \label{Tp}
 T_x^p=\inf\{ j>T_x^{p-1}: S_j=x\}
 \end{align}
 with the convention   $\inf \emptyset =\infty$. 
For $l,n\in{\mathbb{N}}$ let 
 $$R[l,n]=\{S_l,S_{l+1},...,S_n\}$$
  if $n\geq l$
 and  $R[l,n]=\emptyset$ if  $l>n$  and $R[0,-1]=\emptyset$. 
The inner boundary of the random walk range $R[l,n]$ is denoted simply by $\partial R[l,n]$ as it is for $R(n)$. 
 It is noted that  $T_x=T_x^0$ for $x\neq S_0$ and $T_{x}=T_{x}^1$ if  $x=S_0$. Also,  $R(n)= R[0,n]$.
 
In the proofs given in the remainder of this paper we denote contextual constants by $C$ or $c$.  
In addition, $\lceil a \rceil$ denotes the smallest integer $n$ with $n\ge a$, 
and $A^c$ denotes a complementary set of $A$. 
\section{The upper bound in Theorem \ref{m1}}
Here, we prove the following proposition. 
\begin{prop}\label{po}
For $d\ge2$
$$\limsup_{n\to \infty} \frac{M(n)}{\log n}\le \beta_d\quad\text { a.s.}
$$
\end{prop}
Unlike the proof of the lower bound below, the proof of Proposition \ref{po}   will be performed independently of  the dimension $d$.  
As mentioned above, neither $M(n)$ nor $\Theta_n(\delta)$ is  monotone in $n\in \mathbb{N}$. 
To mitigate this issue, we introduce  the random variables.  
For $\beta>0$, we set 
\begin{align*}
\tilde{\Theta}_n(\beta)=\sharp \{ x\in\partial R(T_x^{\lceil \beta\log  n/2\rceil }) :K(n,x)\ge \lceil \beta \log \frac{n}{2}\rceil \}
\end{align*}
($T_x^p$ is defined by (\ref{Tp})). 
This is a modification of $\Theta_n(\beta/\beta_d)$ made by relaxing the constraint of being on the inner boundary.  
Note that $\tilde\Theta_n(\beta)$  vanishes for all sufficiently large $n\in \mathbb{N}$ if $\beta>\beta_d$.

 \begin{lem}\label{theta}
For $\beta>0$ there exists $C>0$ such that for any $n\in \mathbb{N}$
\begin{align*}
E[\tilde{\Theta}_n(\beta)]\le Cn^{1-\frac{\beta}{\beta_d}}.
\end{align*}
\end{lem}
\begin{proof}
First we introduce the elementary property. 
For any intervals $I_0$, $I_1$, $I_2\subset {\mathbb N}\cup \{0\}$ with $I_0 \subset I_1 \subset I_2$, it holds that 
\begin{align}\label{el*}
R(I_0)\cap \partial R (I_2)  \subset \partial R(I_1).
\end{align}
Note that we can write 
\begin{align}\label{by}
\tilde{\Theta}_n(\beta)
=\sum_{l=0}^n 1_{B_{l,n}},
\end{align}
where 
\begin{align*}
B_{l,n}=\{S_l \in R(l-1)^c \cap \partial R(T_{S_l}^{\lceil \beta \log n/2\rceil}),
K(n,S_l)\ge \lceil \beta \log  \frac{n}{2} \rceil \}.
\end{align*}
Since $K(l-1,S_l)=0$ on $\{S_l \in R(l-1)^c\}$, 
for $l\le n$
\begin{align*}
P(B_{l,n})=
&P(\sum_{j=0}^n1_{\{S_j=S_l\}}\ge \lceil \beta \log \frac{n}{2}\rceil, S_l \in R(l-1)^c \cap \partial R(T_{S_l}^{\lceil \beta \log n/2\rceil}) )\\
=&P(\sum_{j=l}^n1_{\{S_j=S_l\}}\ge \lceil \beta \log \frac{n}{2}\rceil, S_l \in R(l-1)^c \cap \partial R(T_{S_l}^{\lceil \beta \log n/2\rceil}) )\\
\le &P(\sum_{j=l}^n1_{\{S_j=S_l\}}\ge \lceil \beta \log \frac{n}{2}\rceil,
S_l\in \partial R[l,T_{S_l}^{\lceil \beta \log n/2\rceil}])\\
=&P(K(n-l,0)\ge \lceil \beta \log \frac{n}{2}\rceil, 0 \in  \partial R(T_{0}^{\lceil \beta \log n/2\rceil-1}) ).
\end{align*}
Here, the inequality comes from (\ref{el*}) 
with $I_0=\{ l \}$, $I_1=[l, T_{S_l}^{\lceil \beta \log n/2\rceil}]$ 
and $I_2=[0, T_{S_l}^{\lceil \beta \log n/2\rceil}]$. 
The last equality follows from the Markov property and the translation invariance for $S_l$. 
In addition, by applying the Markov property repeatedly, we obtain
\begin{align*}
&P(K(n-l,0)\ge \lceil \beta \log \frac{n}{2}\rceil, 0 \in  \partial R(T_{0}^{\lceil \beta \log n/2\rceil-1}) )\\
\le &P( T_0^{\lceil \beta \log n/2\rceil-1}<\infty, 0 \in  \partial R(T_{0}^{\lceil \beta \log n/2\rceil-1}) )\\
= &P(\cup_{b\in {\cal N}(0)}\{T_0^{\lceil \beta \log n/2\rceil-1 }< T_b\} )\\
\le &2dP(T_0<T_b)^{\lceil \beta \log n/2\rceil-1}
\le C n^{-\frac{\beta}{\beta_d}}.
\end{align*}
Hence, the assertion holds by $E[\tilde{\Theta}_n(\beta)]\le n\max_{1\le l \le n} P(B_{l,n})$ which follows from  (\ref{by}). 
\end{proof}

\begin{proof}[Proof of Proposition \ref{po}]
Since $M(n)$ is not monotonically increasing in $n$, 
we instead first consider $\tilde{M}(n)=\max_{l\le n}M(l)$. 
If $\tilde{M}(n)>m$, there exist $x_1\in \mathbb{Z}^d$ and  $n_1\le n$ such that 
$\tilde{M}(n)=M(n_1)=K(n_1,x_1)$ and  $x_1\in \partial R(n_1)$. 
Therefore, for such $n_1$, $m$ and $x_1$, it holds that $T_{x_1}^m\le n_1$ and hence $x_1\in\partial R(n_1)\cap \partial R(T_{x_1}^m)$ holds.  
Further, $K(n_1,x_1)\le K(n,x_1)$. 
Accordingly, we have 
\begin{align*}
&P(\tilde{M}(n)\ge \lceil \beta \log \frac{n}{2}\rceil)\\
\le & P(\cup_{x \in \mathbb{Z}^2 }\{x \in \partial R(T_{x}^{\lceil \beta \log n/2\rceil}), K(n,x)\ge \lceil \beta \log \frac{n}{2}\rceil \} ).
\end{align*} 
Thus, by Lemma \ref{theta} and the Chebyshev inequality, we obtain
\begin{align*}
P(\tilde{M}(n)\ge \lceil \beta \log \frac{n}{2}\rceil)
\le P(\tilde{\Theta}_n(\beta)\ge1)
\le Cn^{1-\frac{\beta}{\beta_d}}.
\end{align*} 
By using the Borel-Cantelli lemma  for any $\beta>\beta_d$, we can show that the events 
$\{\tilde{M}(2^k)>\beta \log 2^{k-1}\}$  
happen only finitely often with probability one. 
Therefore, it holds that for any $\beta>\beta_d$, 
\begin{align}\label{guu*}
\limsup_{k\to\infty} \frac{\tilde{M}(2^k)}{\beta \log 2^{k-1} }\le 1\quad\text{ a.s. }
\end{align}
Now we consider $M(n)$. 
For  any $k$, $n\in{\mathbb N}$ with $2^{k-1}\le n<2^k$ we have
$$ \frac{M(n)}{\log n}\le\frac{\tilde{M}(2^k)}{\log 2^{k-1}},$$ 
and so with (\ref{guu*}), for any $\beta>\beta_d$ 
\begin{align*}
\limsup_{n\to\infty} \frac{M(n)}{\beta\log n}
\le 1 \quad\text{ a.s. }
\end{align*}
Therefore, the proof is completed. 
\end{proof}

\section{The lower bound in Theorem \ref{m1}}
\subsection{Reduction of the lower bound of Theorem \ref{m1} to key lemmas}
Our goal in this section is to prove the following Proposition \ref{po1}. 

\begin{prop}\label{po1}
For $d\ge2$
\begin{align}\label{kii}
\liminf_{n\to \infty} \frac{M(n)}{ \log n}\ge \beta_d \quad\text{ a.s.}
\end{align}
\end{prop}
Unlike Sections $4.3$ and $4.4$, the argument of Sections $4.1$ and $4.2$ will be performed independently of  the dimension $d$.  
In what follows, we discuss the proof for each fixed $\beta<\beta_d$. 
Let $$u_n=\lceil \exp(n^2)\rceil.$$ 
In this section, we will reduce the proof of Proposition \ref{po1} to two key lemmas given below 
(Lemmas \ref{hh+} and \ref{hh}). 
For $b\in {\cal N}(0)$ and 
$A \subset \mathbb{Z}^d$, we define $\partial_b A$ as
$$\partial_b A:=\{x\in A : x+b \not\in A \}.$$ 
We can extend the property (\ref{el*}) in the following way: 
for any intervals $I_0$, $I_1$, $I_2\subset {\mathbb N}\cup \{0\}$ with $I_0 \subset I_1 \subset I_2$, it holds that 
\begin{align}\label{el}
R(I_0)\cap \partial_b R (I_2)  \subset \partial_b R(I_1)\subset \partial R(I_1).
\end{align}
Let us define $\tilde{Q}_n$ as
\begin{align*}
\tilde{Q}_n:=\sharp \{x\in  R(u_{n-1}) \cap \partial_b R(u_n),
T_{x}^{\lceil\beta n^2\rceil}\le u_{n-1}\}. 
\end{align*}
We begin by providing a sufficient condition for the inequality (\ref{kii}) asserted in Proposition \ref{po1} to be true by means of $\tilde{Q}_n$. 
\begin{lem}\label{fu}
If 
\begin{align}\label{iii}
P(\tilde{Q}_n\ge 1 \quad\quad  \text{ for all but finitely many }n)=1, 
\end{align}
for any $\beta\in (0,\beta_d)$ then the inequality (\ref{kii}) holds. 
\end{lem}
\begin{proof}
Set
$$L(n):=\max_{x \in  R(u_{n-1}) \cap \partial_b R(u_n)}K(u_{n-1},x).$$
Note that $L(n) \ge \beta n^2$ if $\tilde{Q}_n\ge 1$. 
Hence, it holds that
\begin{align*}
P(L(n)\ge\beta n^2\quad\quad  \text{ for all but finitely many }n)=1.
\end{align*}
Moreover, by (\ref{el}), for any $m$, $n \in{\mathbb{N}}$ with $u_{m-1}\le n < u_m$  
we have 
\begin{align}\label{ooo}
 R(u_{m-1}) \cap \partial_b R(u_m)
\subset \partial R(n)
\end{align}
and hence $L(m) \le \max_{x\in \partial R(n)}K(u_{m-1},x)\le \max_{x\in \partial R(n)}K(n,x)  \le M(n)$.   
Therefore, we conclude that for any $\beta<\beta_d$
\begin{align*}
\liminf_{n\to\infty} \frac{M(n)}{\beta \log n}\ge
\liminf_{m\to\infty} \frac{L(m)}{\beta \log u_m}\ge
1\quad \text{ a.s.,}
\end{align*}
as per (\ref{kii}) and as desired. 
\end{proof}
In order to verify the condition (\ref{iii}),
we introduce a new quantity $Q_n$ by modifying the definition of $\tilde{Q}_n$. 
To do this, we first introduce several notions. 
Set 
$T_{x,l}^0:=\inf\{j\ge l: S_j=x\}$ and 
$T_{x,l}^p:=\inf\{j>T_{x,l}^{p-1}:S_j=x\}$.
Note that $T_{x,0}^a=T_x^a$ for any $a\in{\mathbb{N}}$ and $x\in \mathbb{Z}^d$, 
and that $T_{S_l}^p=T_{S_l,l}^p$ holds for each $p$ on the event $S_l \notin R(l-1)$. 
Note that $T_{S_l,l}^p$ is a stopping time while $T_{S_l}^p$ is not. 
Although we can state the key lemmas without using this notion, 
we introduce it for later use. 
For $k\in {\mathbb{N}}$, let 
$h_k=\beta\log P(T_0<T_b\wedge k)+1$. 
Since $\lim_{k\to \infty}h_k=1-\beta/\beta_d>0$, we have $h_k>0$ 
for all sufficiently large $k$. 
We fix such a $k$ and simply denote $h_k$ by $h$ hereafter. 
Let $$I_n:=[\frac{u_{n-1}}{n^2}, u_{n-1}-k \lceil \beta_d n^2\rceil]\cap {\mathbb{N}}.$$ 
For any $l\in I_n$, we introduce the events $E_{l,n}$ and $A_{l,n}$ defined by 
\begin{align*}
E_{l,n}:=\{ T_{S_l,l}^{j}-T_{S_l,l}^{j-1}<k \text{ for any }1\le j\le \lceil \beta n^2\rceil \},
\end{align*}
and
\begin{align*}
A_{l,n}:=\{S_l \in R(l-1)^c \cap \partial_b R(u_{n}) \}\cap E_{l,n}.
\end{align*}
Then, we set
\begin{align}\label{yy}
Q_n:=\sum_{l \in I_n }1_{A_{l,n}}. 
\end{align}
Although $I_n$, $A_{l,n}$ and $Q_n$ depend on the choice of parameters $\beta$ and $k$, 
we do not indicate such dependence explicitly by symbols. 
By the definition of $I_n$,  
$I_n \subset [0,u_{n-1}] \cap (\mathbb{N} \cup \{0\})$ and 
$l+k \lceil \beta n^2\rceil \le u_{n-1}$ hold. 
These facts imply $Q_n\le \tilde{Q}_n$. 
As we will see, the verification of condition (\ref{iii}) is reduced to the following two estimates for $Q_n$. 
\begin{lem} \label{hh+}
Let $\beta<\beta_d$ and take $k\in{\mathbb{N}}$ so that $h=h_k>0$ as above.  
Then, there exists $c>0$ such that for any $n\in{\mathbb{N}}$, the following hold:  

(i) When $d=2$,
\begin{align*}
EQ_n \ge \frac{c\exp(hn^2-2n)}{n^4}.
\end{align*}

(ii) When $d \ge 3$,
\begin{align*}
EQ_n \ge c\exp(hn^2-2n).
\end{align*}
\end{lem}

\begin{lem} \label{hh}
Let $\beta<\beta_d$ and take $k\in{\mathbb{N}}$ so that $h=h_k>0$ as above.  
Then, there exists $C>0$ such that for any $n\in{\mathbb{N}}$, the following hold: 

(i) When $d=2$,
\begin{align*}
\mathrm{Var} (Q_n) \le C\bigg(\frac{\exp(hn^2-2n)}{n^4}\bigg)^2 \frac{\log n}{n^2}.
\end{align*}

(ii) When $d \ge 3$,
\begin{align*}
\mathrm{Var} (Q_n) \le C\exp(2hn^2-4n)\times \frac{1}{n^{10}}.
\end{align*}
\end{lem}
Now we give a proof of Proposition \ref{po1} by assuming Lemmas \ref{hh+} and \ref{hh} are true.  
\begin{proof}[Deduction of Proposition \ref{po1} from Lemmas \ref{hh+} and \ref{hh}]
If Lemmas \ref{hh+} and \ref{hh} hold, 
then we only need to prove the assumption of Lemma \ref{fu} to obtain Proposition \ref{po1}.  
Take $k\in {\mathbb{N}}$ and $h>0$ as above. 
By the Chebyshev inequality, we have
\begin{align}\label{ine1}
P(|Q_n-EQ_n|>\frac{EQ_n}{2})\le \frac{4\mathrm{Var} (Q_n)}{(EQ_n)^2}.
\end{align}
By Lemmas \ref{hh+} and \ref{hh}, we can see that there exists $C>0$ 
such that the following is true: 
 \begin{align*}
\frac{\mathrm{Var} (Q_n)}{(EQ_n)^2} 
\begin{cases}
&\le \frac{C\log n}{n^2}\quad \quad\quad \quad \text{if } d=2,\\
&\le \frac{C}{n^{10}} \quad \quad \quad \quad\quad \text{ if } d \ge 3.
\end{cases}
\end{align*}
As a result, the right hand side of (\ref{ine1}) is summable. 
Since $|b-a|\le \frac{a}{2}$ implies $b\ge \frac{a}{2}$ 
for $a,b\ge0$, the Borel-Cantelli lemma yields 
\begin{align}\label{rrr}
P(Q_n\ge\frac{1}{2}EQ_n\quad \quad \text{ for all but finitely many }n)=1.
\end{align} 
Since $h>0$, Lemma \ref{hh+} implies $EQ_n \ge 2$ for all sufficiently large $n$. 
Hence, the assertion of Lemma \ref{fu} holds by combining this fact with (\ref{rrr}). 
\end{proof}

\subsection{Preparations for the proof of Lemmas \ref{hh+} and \ref{hh}}
In this section, we estimate $P(A_{l,n})$ using the strong Markov property. 
For later use, we will consider more general events than $A_{l,n}$. 
For any $n', l , \tilde{n}\in {\mathbb{N}} \cup \{0\}$ with $n'\le l \le \tilde{n}$ and $n\in {\mathbb{N}}$, let
\begin{align*}
F_{n', l, \tilde{n},n}=\{S_l \in R[n',l-1]^c \cap \partial_b R[n',\tilde{n}]\} \cap E_{l,n},
\end{align*}
which we  will sometimes denote $F(n', l, \tilde{n},n)$ for typographical reasons. 
Note that $F_{0,l,u_n,n}=A_{l,n}$ holds. 
\begin{lem}\label{subs}
There are constants $c$, $C>0$ such that for any $n', l , \tilde{n}\in {\mathbb{N}} \cup \{0\}$ with $n'\le l \le \tilde{n}$ and $n,a \in {\mathbb{N}}$
with $l+k\lceil\beta_d n^2\rceil \le a\le \tilde{n}$
\begin{align*}
P(F_{n', l, \tilde{n},n} )
\le P(T_0<T_b\wedge k )^{\lceil \beta n^2\rceil}P(T_0 \wedge T_b >l-n')\times P(T_b>\tilde{n}-a)
\end{align*}
and 
\begin{align*}
P(F_{n', l, \tilde{n},n})
\ge P(T_0<T_b\wedge k )^{\lceil \beta n^2\rceil}P(T_0 \wedge T_b >l-n')\times P(T_b>\tilde{n}).
\end{align*}
\end{lem}
\begin{proof}
We first remark that $l<T_{S_l,l}^{\lceil\beta n^2\rceil}$ 
holds and, hence, ${\cal F}(l) \subset {\cal F}(T_{S_l,l}^{\lceil\beta n^2\rceil})$. 
By taking a conditional expectation with respect to ${\cal F}(T_{S_l,l}^{\lceil\beta n^2\rceil})$, we obtain
\begin{align}\notag
P(F_{n', l, \tilde{n},n})
=&E[1{\{S_l\in R[n',l-1]^c \cap \partial_b R[n',T_{S_l,l}^{\lceil\beta n^2\rceil}] \}\cap E_{l,n}}\\
\label{t*}
&\times P(S_l\in \partial_b R[T_{S_l,l}^{\lceil\beta n^2\rceil},\tilde{n}]
|{\cal F}(T_{S_l,l}^{\lceil\beta n^2\rceil} ))].
\end{align}
On the event $E_{l,n}$, we have
\begin{align*}
0\le T_{S_l,l}^{\lceil\beta n^2\rceil} \le  k\lceil \beta_d n^2 \rceil+T_{S_l,l}^0.
\end{align*}
Since $l=T_{S_l,l}^0$, our choice of  $a$ and this inequality imply
\begin{align}\label{rr*}
0\le T_{S_l,l}^{\lceil\beta n^2\rceil} \le a.
\end{align}
The Markov property and the translation invariance for $S_l$ yield
\begin{align}\label{t*1}
 P(S_l\in \partial_b R[T_{S_l,l}^{\lceil\beta n^2\rceil},\tilde{n}]
|{\cal F}(T_{S_l,l}^{\lceil\beta n^2\rceil} ))
= P(0 \in \partial_b R(\tilde{n}-t))|_{t=T_{S_l,l}^{\lceil\beta n^2\rceil}}.
\end{align}
Substituting (\ref{t*1}) for (\ref{t*}) 
and keeping (\ref{rr*}) in mind, we obtain
\begin{align}\notag
P(F_{n', l, \tilde{n},n})
\le&P(\{S_l\in R[n',l-1]^c\cap \partial_b R[n',T_{S_l,l}^{\lceil\beta n^2\rceil}]\}
\cap E_{l,n})\\
\label{t1}
&\times P(0 \in \partial_b R(\tilde{n}-a)),
\end{align}
and 
\begin{align}\notag
P(F_{n', l, \tilde{n},n})
\ge&P(\{S_l\in R[n',l-1]^c \cap \partial_b R[n',T_{S_l,l}^{\lceil\beta n^2\rceil}]\}
\cap E_{l,n})\\
\label{t2}
&\times P(0 \in \partial_b R(\tilde{n})).
\end{align}
Thus, the proof is reduced to the estimate of the common first factor in the right hand side of (\ref{t1}) and (\ref{t2}). 
If we take a conditional expectation with respect to ${\cal F}(l)$, 
by the Markov property and the translation invariance for $S_l$, we obtain
\begin{align}\notag
&P(\{S_l\in R[n',l-1]^c \cap \partial_b R[n',T_{S_l,l}^{\lceil\beta n^2\rceil}]\}
\cap E_{l,n})\\
\label{t3}
= &P(S_l\in R[n',l-1]^c \cap \partial_b R[n',l])
\times P(\{0 \in \partial_b R(T_{0}^{\lceil \beta n^2\rceil})\}
\cap E_{0,n}).
\end{align}
By the choice of $h$, it holds that
\begin{align}
\label{t4}
P(\{0 \in \partial_b  R(T_{0}^{\lceil \beta n^2\rceil})\}
\cap E_{0,n})
=P(T_0<T_b\wedge k )^{\lceil \beta n^2\rceil},
\end{align}
where there exist $c$, $C>0$ such that for any $n\in \mathbb{N}$
\begin{align}\label{t4*}
c\exp((h-1)n^2)
\le P(T_0<T_b\wedge k )^{\lceil \beta n^2\rceil}\le C\exp((h-1)n^2).
\end{align}
By considering a time-reversal, we obtain
\begin{align}
\label{t5}
 P(S_l \in R[n',l-1]^c \cap \partial_b R[n',l])
=P(0\in R[1,l-n']^c \cap \partial_b R(l-n')).
\end{align}
In addition, for $m\in \mathbb{N}$,  we have 
$P(0\in R[1,m]^c \cap \partial_b R(m))=P(T_0 \wedge T_b >m)$, and 
$P(0 \in \partial_b R(m))=P(T_b>m)$. 
Therefore, by (\ref{t1}), (\ref{t2}), (\ref{t3}), (\ref{t4}) and (\ref{t5}), the desired formulas hold.
\end{proof}
\begin{rem}\label{pl}
We substitute $T_{S_l,l}^{\lceil\beta n^2\rceil}$ for $\tilde{n}$ of $F_{n', l, \tilde{n},n}$. 
That is, for any $n', l \in {\mathbb{N}} \cup \{0\}$ with $n'\le l $,  we write
\begin{align*}
F(n', l, T_{S_l,l}^{\lceil\beta n^2\rceil},n)
=\{S_l \in R[n',l-1]^c \cap \partial_b R[n',T_{S_l,l}^{\lceil\beta n^2\rceil}]\} \cap E_{l,n} .
\end{align*}
By the same argument, we obtain the following: for any 
$n', l \in {\mathbb{N}} \cup \{0\}$ with $n'\le l $ 
\begin{align}\label{ppp*}
P(F(n', l, T_{S_l,l}^{\lceil\beta n^2\rceil},n) )
\le P(T_0<T_b\wedge k )^{\lceil \beta n^2\rceil}P(T_0 \wedge T_b >l-n').
\end{align}
(See the argument after (\ref{t3}).)
\end{rem}

\begin{cor}\label{al}
For any $l\in I_n$ and all sufficiently large $n\in \mathbb{N}$ 
with $u_n/n^{11}\le u_n-u_{n-1}$, 
\begin{align}\notag
P(A_{l,n} )
\le &P(T_0 \wedge T_b >l)
\times  P(T_0<T_b \wedge k )^{\lceil\beta n^2\rceil}\\
\label{hy*}
&\times P(T_b >\frac{u_n}{n^{11}})
\end{align}
and 
\begin{align}\notag
P(A_{l,n})
\ge &P(T_0 \wedge T_b >l)
\times P(T_0<T_b \wedge k )^{\lceil\beta n^2\rceil}\\
\label{hy}
&\times P(T_b >u_n).
\end{align}
\end{cor}
\begin{proof}
For $l\in I_n$
\begin{align*}
l=T_{S_l,l}^0 \le u_{n-1}- k\lceil \beta_d n^2\rceil
\end{align*}
holds. 
Therefore, by using Lemma \ref{subs} with $F_{0,l,u_n,n}=A_{l,n}$ and substituting $u_{n-1}$ for  $a$ in the assumption of Lemma \ref{subs} we obtain (\ref{hy}) and
\begin{align*}
P(A_{l,n} )
\le &P(T_0 \wedge T_b >l)
\times  P(T_0<T_b \wedge k )^{\lceil\beta n^2\rceil}\\
&\times P(T_b >u_n-u_{n-1}).
\end{align*}
Therefore, we obtain (\ref{hy*}) for all sufficiently large $n\in \mathbb{N}$.
\end{proof}

\subsection{Proof of Lemmas \ref{hh+} and \ref{hh} for $d\ge3$}\label{e1} 
By Corollary \ref{al},  
we obtain the following estimate of $P(A_{l,n})$.  
\begin{lem}
There exist constants $C$, $c>0$ such that for any $n\in \mathbb{N}$ and $l\in I_n$
\begin{align}\label{ju}
P(A_{l,n})
\le & C\exp((h-1)n^2)\\
\label{ju*}
P(A_{l,n})
\ge &c\exp((h-1)n^2).
\end{align}
Moreover, for any $n\in \mathbb{N}$ and $l\in I_n$
\begin{align}\notag
P(A_{l,n})
= &P(T_0\wedge T_b=\infty)\times  P(T_0<T_b \wedge k )^{\lceil\beta n^2\rceil}\times P(T_b=\infty)\\
\label{ju**}
&+O(\exp((h-1)n^2-cn^2)). 
\end{align}
\end{lem}
\begin{proof}
Since (\ref{t4*}) and (\ref{ju**}) yield (\ref{ju}) and  (\ref{ju*}), 
we only need to prove (\ref{ju**}). 
First, we introduce some estimates of hitting times. 
Since $P(S_n=0)=O(n^{-\frac{d}{2}})$, we obtain $P(T_0=n)=O(n^{-\frac{d}{2}})$ 
and, hence, for $d\ge 3$ and $b\in {\cal N}(0)$,
\begin{align}\label{g7}
P(n\le T_0< \infty)&\le \sum_{m=n}^{\infty} O(n^{-\frac{d}{2}} )=O( n^{-\frac{d}{2}+1}).
\end{align}
In addition, by the Markov property and the translation invariance for $b'$ we have 
 \begin{align}\label{kkk}
P(T_0 \ge a+1) 
=\frac{1}{2d}\sum_{b'\in {\cal N}(0)}P^{b'}(T_0 \ge a)
=\frac{1}{2d}\sum_{b'\in {\cal N}(0)}P(T_{-b'} \ge a)
=P(T_{b}\ge a),
\end{align}
for $a\in \mathbb{N}$. 
Hence, (\ref{g7}) yields 
\begin{align}\label{g8}
P(n\le T_b<\infty)&=O(n^{-\frac{d}{2}+1}).
\end{align}
Moreover, it holds that 
\begin{align*}
&P(n\le T_0\wedge T_b<\infty)
\le P(\{n\le T_0<\infty\}\cup\{ n\le T_b<\infty\})\\
=&P(n\le T_0<\infty)+P(n\le T_b<\infty),
\end{align*}
and so with (\ref{g7}) and (\ref{g8}),
\begin{align}
\label{g9}
P(n\le T_0\wedge T_b<\infty)&= O(n^{-\frac{d}{2}+1}).
\end{align}
Therefore, by (\ref{g7}), (\ref{g8}) and (\ref{g9}) there exists $c>0$ for any $n\in \mathbb{N}$ and $l \in I_n$
\begin{align}
\label{bb0}
&P(T_0 \wedge T_b>l )
= P(T_0\wedge T_b=\infty)+O(\exp(-cn^2)),\\
\label{bb1}
&P(T_b> u_{n})
=P(T_b=\infty)+O(\exp(-cn^2)),\\
\label{bb2}
&P(T_b >\frac{u_n}{n^{11}})
= P(T_b=\infty)+O(\exp(-cn^2)).
\end{align}
Substituting (\ref{bb0}), (\ref{bb1}) and (\ref{bb2}) for the right hand sides of (\ref{hy*}) and (\ref{hy}), 
we obtain the desired formula. 
\end{proof}
\begin{proof}[Proof of Lemma \ref{hh+} for $d\ge3$]
Since for all sufficiently large $n\in \mathbb{N}$ 
\begin{align}\label{lo}
\sharp I_n\ge u_{n-1}-k\lceil \beta_d n^2\rceil -\frac{u_{n-1}}{n^2}\ge c u_{n-1}
\end{align}
holds, by (\ref{ju*}) and the definition of $Q_n$ in (\ref{yy}), we obtain
\begin{align*}
EQ_n\ge \sharp I_n \times \min_{l\in I_n} P(A_{l,n})
 \ge c\exp(hn^2-2n),
\end{align*}  
as required.  
\end{proof}

To prove Lemma \ref{hh}, 
we decompose $I_n\times I_n$ into three parts $J_{n,j}$, $j = 1,2,3$ defined by
\begin{align*}
&J_{n,1}:=\{(l,l')\in I_n\times I_n:0\le l'-l\le k \lceil \beta_d n^2\rceil \},\\ 
&J_{n,2}:=\{(l,l')\in I_n\times I_n:k \lceil \beta_d n^2\rceil < l'-l\le 2\lceil \frac{u_{n-1}}{n^{10}} \rceil \},\\ 
&J_{n,3}:=\{(l,l')\in I_n\times I_n:2\lceil \frac{u_{n-1}}{n^{10}} \rceil< l'-l \}. 
\end{align*}
For all sufficiently large $n\in \mathbb{N}$, $k \lceil \beta_d n^2\rceil < 2\lceil u_{n-1}/n^{10} \rceil$ holds and hence 
$J_{n,2}$ is non-empty.  
By a simple computation,
\begin{align}\notag
\mathrm{Var} (Q_n)=&EQ_n^2-(EQ_n)^2\\
\notag
\le &2\sum_{l,l' \in I_n,l\le l'}
(E[1_{A_{l,n}}1_{A_{l',n}}]-E[1_{A_{l,n}}]E[1_{A_{l',n}}])\\
\notag
\le &2(\sum_{(l,l') \in J_{n,1}} E[1_{A_{l,n}}]
+\sum_{(l,l')  \in J_{n,2}}
E[1_{A_{l,n}}1_{A_{l',n}}]\\
\label{formula2}
+&\sum_{(l,l') \in J_{n,3}}
(E[1_{A_{l,n}}1_{A_{l',n}}]-E[1_{A_{l,n}}]E[1_{A_{l',n}}])).
\end{align}
\begin{lem}\label{ku}
There exists $C>0$ such that for any $n\in \mathbb{N}$ 
\begin{align}\label{oo}
\sum_{(l,l') \in J_{n,1}} E[1_{A_{l,n}}]\le Cn^2\exp(hn^2-2n) 
\end{align}
and 
\begin{align}\label{o1}
\sum_{(l,l')  \in J_{n,2}}
E[1_{A_{l,n}}1_{A_{l',n}}]
\le C\exp(2hn^2-4n)\times \frac{1}{n^{10}}.
\end{align} 
\end{lem}
\begin{rem}\label{gh}
This Lemma also holds for $d=2$ by the same proof. 
\end{rem}
\begin{proof}
First, we show (\ref{oo}). 
By the definition, we have $\sharp J_{n,1} \le k\lceil \beta_d n^2\rceil u_{n-1}$. 
Thus, (\ref{ju}) yields
\begin{align*}
\sum_{(l,l') \in J_{n,1}} E[1_{A_{l,n}}]
\le \sharp J_{n,1} \times \max_{l\in I_n}P(A_{l,n})
\le Cn^2\exp(hn^2-2n) .
\end{align*} 
Hence, we obtain (\ref{oo}).
To show (\ref{o1}), let us introduce additional notations. 
We define 
\begin{align}
\label{f1}
\tilde{A}_{l,n}:=\{ S_l\in  \partial_b R[l+1,T_{S_l,l}^{\lceil \beta n^2\rceil}] \}
\cap E_{l,n}.
\end{align}
Note that 
\begin{align*}
&\tilde{A}_{l,n}=F(l, l, T_{S_l,l}^{\lceil\beta n^2\rceil},n),\\
&\tilde{A}_{l,n}\cap \{ S_l\in  R(l-1)^c\cap \partial_b R(l) \}\cap \{ S_l\in \partial_b R[T_{S_l,l}^{\lceil \beta n^2\rceil},u_n] \}
=A_{l,n}
\end{align*}
and
\begin{align*}
\tilde{A}_{l,n}\in \sigma \{ S_j-S_l: j\in  [l,T_{S_l,l}^{\lceil \beta n^2\rceil}]\}.
\end{align*}
By Remark \ref{pl}, we obtain 
\begin{align}\label{hj}
P(\tilde{A}_{l,n})
=P(F(l, l, T_{S_l,l}^{\lceil\beta n^2\rceil},n))
\le C\exp((h-1)n^2).
\end{align}
We obtain (\ref{o1}) as follows:
by the definition of $A_{l,n}$ and $\tilde{A}_{l,n}$, 
we have ${A}_{l,n}\subset \tilde{A}_{l,n}$ and ${A}_{l',n}\subset \tilde{A}_{l',n}$.  
In addition, $\tilde{A}_{l',n}$ is independent of $\tilde{A}_{l,n}$ for $(l,l')\in J_{n,2}$. 
Thus,
\begin{align*}
E[1_{A_{l,n}}1_{A_{l',n}}]
\le 
E[1_{\tilde{A}_{l,n}}1_{\tilde{A}_{l',n}}]
=E[1_{\tilde{A}_{l,n}}]E[1_{\tilde{A}_{l',n}}],
\end{align*}   
and so by (\ref{hj}), 
\begin{align*}
\sum_{(l,l')  \in J_{n,2}}
E[1_{A_{l,n}}1_{A_{l',n}}]
= &\sharp J_{n,2}\times
C(\exp((h-1)n^2))^2\\
\le &C\exp(2hn^2-4n)\times \frac{1}{n^{10}}.
\end{align*} 
Therefore, we obtain (\ref{o1}).
\end{proof}
\begin{proof}[Proof of Lemma \ref{hh} for $d\ge3$]
We estimate the last sum appearing in (\ref{formula2}). To this end, set
\begin{align}\label{f2}
A'_{l,n}:=&\{ S_l\in   R(l-1)^c \cap\partial_b R(l+\lceil \frac{u_{n-1}}{n^{10}} \rceil)\} \cap E_{l,n}\\
\label{f6}
A''_{l',n}:=&\{S_{l'}\in R[l'-\lceil \frac{u_{n-1}}{n^{10}} \rceil,l'-1]^c\cap \partial_b R[l'-\lceil \frac{u_{n-1}}{n^{10}} \rceil,u_{n}] \} \cap E_{l',n}.
\end{align}
Note that 
\begin{align*}
A'_{l,n}=F_{0,l,l+\lceil \frac{u_{n-1}}{n^{10}} \rceil,n},\quad
A''_{l',n}=F_{l'-\lceil \frac{u_{n-1}}{n^{10}} \rceil,l',u_n,n}.
\end{align*}
By Lemma \ref{subs}, (\ref{bb0}) and (\ref{bb2}),  
we can estimate $P(A'_{l,n})$ and $P(A''_{l',n})$ as follows: 
for any $l \in I_n$ and all sufficiently large $n\in \mathbb{N}$ with 
$ u_{n-1}/n^{11} \le  (\lceil u_{n-1}/n^{10}\rceil-k\lceil\beta n^2\rceil) \wedge l$ 
\begin{align}\notag
&P(A'_{l,n})=P(F_{0,l,l+\lceil \frac{u_{n-1}}{n^{10}} \rceil,n} )\\
\notag
\le &P(T_0\wedge T_b>l )
\times  P(T_0<T_b \wedge k )^{\lceil\beta n^2\rceil}
\times P(T_b> \frac{u_{n-1}}{(n-1)^{11}}  )\\
\label{dd}
\le&
P(T_b=\infty)P(T_0\wedge T_b=\infty)
P(T_0<T_b \wedge k  )^{\lceil\beta n^2\rceil}
+O(\exp((h-1)n^2-cn^2))
\end{align}
and for any $l' \in I_n$ and all sufficiently large $n\in \mathbb{N}$ with $u_n/n^{11}\le u_n-(l'+k\lceil\beta n^2\rceil)$ and $u_{n-1}/n^{11} \le\lceil u_{n-1}/n^{10}\rceil\le l'$
\begin{align}\notag
 &P(A''_{l',n})=P(F_{l'-\lceil \frac{u_{n-1}}{n^{10}} \rceil,l',u_n,n} )\\
\notag
\le &P(T_0\wedge T_b>\lceil \frac{u_{n-1}}{n^{10}} \rceil )
\times  P(T_0<T_b \wedge k)^{\lceil\beta n^2\rceil}
\times P(T_b> \frac{u_{n}}{n^{11}} )\\
\label{ddd}
\le &P(T_b=\infty)P(T_0\wedge T_b=\infty)
P(T_0<T_b \wedge k )^{\lceil\beta n^2\rceil}
+O(\exp((h-1)n^2-cn^2)).
\end{align} 
Therefore, by (\ref{ju**}), (\ref{dd}) and (\ref{ddd}), we obtain 
\begin{align}\notag
&\sum_{(l,l')  \in J_{n,3}}
(E[1_{A_{l,n}}1_{A_{l',n}}]-E[1_{A_{l,n}}]E[1_{A_{l',n}}])\\
\notag
\le &\sum_{(l,l')  \in J_{n,3}}
(E[1_{A'_{l,n}}]E[1_{A''_{l',n}}]-E[1_{A_{l,n}}]E[1_{A_{l',n}}])\\
\notag
\le &\sum_{(l,l') \in J_{n,3}}
(P(T_b=\infty)^2P(T_0\wedge T_b=\infty)^2P(T_0<T_b \wedge k )^{2\lceil\beta n^2\rceil}\\
\notag
&-P(T_b=\infty)^2P(T_0\wedge T_b=\infty)^2P(T_0<T_b \wedge k )^{2\lceil\beta n^2\rceil}\\
\notag
&+ O(\exp(2(h-1)n^2-cn^2) ))\\
\label{pp3}
\le &C(\exp(hn^2-2n))^2\times\exp(-cn^2).
\end{align}
By (\ref{oo}), (\ref{o1}) and (\ref{pp3}),  
the right hand side of (\ref{formula2}) is bounded by 
$\exp(2hn^2-4n)\times 1/n^{10}$. 
This completes the proof of Lemma \ref{hh} for $d\ge3$. 
\end{proof}

\subsection{Proof of Lemmas \ref{hh+} and \ref{hh} for $d=2$}\label{e2}
First, we state a lemma that is important for our proof of Lemma \ref{hh}. 
\begin{lem}
There exists $C>0$ such that for any $n\in \mathbb{N}$ and $x\in {\mathbb{Z}^{2}}$  
with $0<|x|< n\sqrt{u_{n-1}}$
 \begin{align}\label{fo3*}
P(T_0\wedge T_x> \lceil \frac{u_n}{n} \rceil) \le \frac{\pi}{(n+1)^2}+C\frac{\log n}{n^4}.
\end{align}
Moreover, it holds that for $b \in{\cal N}(0)$
 \begin{align}\label{fo3}
P(T_b\wedge T_{x+b}> \lceil \frac{u_n}{n} \rceil) \le \frac{\pi}{(n+1)^2}+C\frac{\log n}{n^4}.
\end{align}
\end{lem}
\begin{proof}
We prove only the first claim since the second one follows from (\ref{fo3*}) by a similar observation as in (\ref{kkk}).  
Decomposing the whole event by means of the last exit time from $\{0,x\}$ by time $\lceil u_n/n\rceil$, we obtain
\begin{align}\notag
1=&\sum_{k=0}^{ \lceil u_n/n\rceil}P(S_{k}=0)P^0(0,x \notin R[1,\lceil \frac{u_n}{n} \rceil-k])\\
\label{r1}
+&\sum_{k=0}^{\lceil u_n/n \rceil}P(S_{k}=x)P^x(0,x \notin R[1,\lceil \frac{u_n}{n} \rceil-k]).
\end{align}
By  the local central limit theorem (see Theorem $1.2.1$ in \cite{Law}), there exists 
$c>0$ such that for any $k\in \mathbb{N}\cup \{0\}$, $x\in {\mathbb{Z}^2}$
\begin{align*}
\begin{cases}
&\displaystyle{P(S_k=x) \ge \frac{2}{\pi k} \exp(-\frac{|x|^2}{k})-\frac{c}{k^2}}\quad \quad\text{ if } k \rightleftharpoons x\\
&\displaystyle{P(S_k=x) =0} \quad \quad \quad \quad \quad \quad \quad \quad\text{ if } k+1 \rightleftharpoons x,
\end{cases}
\end{align*}
where $k\rightleftharpoons x=(x_1,x_2) (\in {\mathbb{Z}^2})$ means that $k+x_1+x_2$ is even.  
Let 
\begin{align*}
\gamma(n)=P(T_0\wedge T_x>\lceil \frac{u_n}{n}\rceil)
=P(0,x \notin R[1,\lceil \frac{u_n}{n} \rceil]).
\end{align*}
By the invariance property of $S_n$ under an isomorphism of ${\mathbb{Z}^2}$, 
for $a\le \lceil u_n/n\rceil$ 
$$\gamma(n) \le P^0(0,x \notin R[1,a])
=P^0(-x,0 \notin R[1,a])
=P^x(0,x \notin R[1,a]).$$ 
Note that 
$\sum_{k=1}^{\lceil  u_n/n \rceil}\frac{2}{\pi k} 1_{\{k \rightleftharpoons x\}}
=\sum_{m=1}^{\lceil  u_n/n \rceil/2}\frac{1}{\pi m}>\frac{1}{ \pi} \log (1+\lceil  u_n/n \rceil/2)$ holds. 
Then, by (\ref{r1}) we obtain  
\begin{align*}
1\ge &\bigg(\sum_{k=1}^{\lceil u_n/n \rceil}
(\frac{2}{\pi k} -\frac{c}{k^2})1_{\{k \rightleftharpoons 0\}}
+\sum_{k=1}^{\lceil u_n/n\rceil}
\bigg(\frac{2}{\pi k} \exp(-\frac{|x|^2}{k})-\frac{c}{k^2}\bigg)
1_{\{k \rightleftharpoons x\}}
\bigg)\gamma(n)\\
\ge &\bigg(\frac{n^2}{\pi}-\frac{\log n}{\pi}-c
+\sum_{k=\lceil n|x|^2\rceil}^{\lceil u_n/n \rceil}\frac{2}{\pi k} \exp(-\frac{|x|^2}{k})1_{\{k \rightleftharpoons x\}}
\bigg)\gamma(n)\\
\ge &\bigg(\frac{n^2}{\pi}-\frac{\log n}{\pi}-c
+\sum_{k=\lceil n|x|^2\rceil}^{\lceil u_n/n \rceil}\frac{2}{\pi k} \exp(-\frac{1}{n})1_{\{k \rightleftharpoons x\}}
\bigg)\gamma(n)\\
\ge &\bigg(\frac{n^2}{\pi}-\frac{\log n}{\pi}-c
+\sum_{k=n^3 u_{n-1}}^{\lceil  u_n/n \rceil}\frac{2}{\pi k} \exp(-\frac{1}{n})1_{\{k \rightleftharpoons x\}}
\bigg)\gamma(n)\\
\ge &\frac{1}{\pi}\bigg(n^2-\log n-c
+n^2-(n-1)^2-4\log n -(1-\exp(-\frac{1}{n}))2n\bigg)
\gamma(n)\\
\ge &\frac{1}{\pi}\bigg((n+1)^2-5\log n -c\bigg)
\gamma(n).
\end{align*}
Thus the assertion follows from an easy rearrangement.
\end{proof}

To prove Lemma \ref{hh+}, we first introduce the following lemma. 
\begin{lem}
There exist constants $C$, $c>0$ such that for any $n\in \mathbb{N}$ and $l\in I_n$
\begin{align}
\label{lu}
P(A_{l,n})&\le \frac{C\exp((h-1)n^2)}{n^4},\\
\label{lu*}
P(A_{l,n})&\ge \frac{c\exp((h-1)n^2)}{n^4}.
\end{align}
In addition, for any $n\in \mathbb{N}\cap \{1\}^c$ and $l\in I_n$
\begin{align}
\label{q1}
P(A_{l,n})= \frac{\pi^2 P(T_0<T_b \wedge k )^{\lceil\beta n^2\rceil} }{2n^2(n-1)^2}+O(\frac{\exp((h-1)n^2)\log n}{n^6}).
\end{align}
\end{lem}
\begin{proof}
Since (\ref{t4*}) and (\ref{q1}) yield (\ref{lu}) and (\ref{lu*}), 
we only prove (\ref{q1}). 
First we introduce the following estimates: for any $M>0$ there exists $C>0$ such that for any $n\in \mathbb{N}$
\begin{align}
\label{fo1}
\frac{\pi}{n^2}-C\frac{\log n}{n^4} \le P(T_b>\frac{u_n}{n^M})\le \frac{\pi}{n^2}+C\frac{\log n}{n^4},\\
\label{fo2}
\frac{\pi}{2n^2}-C\frac{\log n}{n^4} \le P(T_0\wedge T_b> \frac{u_n}{n^M})\le \frac{\pi}{2n^2}+C\frac{\log n}{n^4}.
\end{align}
Since we know  
$$P(T_0> n)=\frac{\pi}{\log n}+O\bigg(\frac{1}{(\log n)^2}\bigg)$$
(see $(2.5)$ in \cite{dvo1}), 
the assertion (\ref{fo1}) follows by a simple calculation of (\ref{kkk}). 
For the latter assertion, we already know a weaker estimate of (\ref{fo2}) involving only the leading term. 
(See Lemma $3.3$ in \cite{okada}.) 
We can obtain the error term of (\ref{fo2}) by modifying the proof in \cite{okada} 
along the argument in \cite{dvo1} in a straightforward way. 
Thus, we omit the proofs of (\ref{fo1}) and (\ref{fo2}).
From (\ref{fo1}) and (\ref{fo2}), 
we already have estimates of each term in (\ref{hy*}) and (\ref{hy}). 
Indeed, for any $l\in I_n$ and all sufficiently large $n\in \mathbb{N}$, 
(\ref{fo1}) yields
\begin{align}\label{rw}
&P(T_b>\frac{u_{n}}{n^{11}} )=\frac{\pi}{n^2}+O(\frac{\log n}{n^4}),\\
\notag
&P(T_b> u_n)=\frac{\pi}{n^2}+O(\frac{\log n}{n^4}).
\end{align}
Since $l \in I_n$, (\ref{fo2}) implies
\begin{align}\label{n1}
P(T_0 \wedge T_b >l)
=\frac{\pi}{2(n-1)^2}+O(\frac{\log n}{n^4}).
\end{align}
Therefore, by substituting these estimates for the right hand sides of (\ref{hy*}) and (\ref{hy}) 
we obtain the desired formula. 
\end{proof}
\begin{proof}[Proof of Lemma \ref{hh+} for $d=2$]
Recall (\ref{lo}).  
Then, (\ref{yy}) and (\ref{lu*}) yield
\begin{align*}
EQ_n \ge \sum_{l\in I_n}
\frac{ c \exp((h-1)n^2) }{n^4}
\ge \frac{c\exp(hn^2-2n)}{n^4}
\end{align*}
for any $n\in \mathbb{N}$, as desired. 
\end{proof}

\begin{proof}[Proof of Lemma \ref{hh} for $d=2$]
By the same argument for $d\ge 3$, we obtain (\ref{formula2}) for $d=2$. 
We consider the estimate of the right hand side of (\ref{formula2}). 
The first term and the second term of  the right hand side of (\ref{formula2}) are already estimated by Lemma \ref{ku}. 
(Note Remark \ref{gh}.) 
To estimate the third term, 
we will give a uniform upper bound of $P(A_{l,n}\cap A_{l',n})$ for $(l,l')\in J_{n,3}$. 
Here, uniform means that the bound is independent of the choice of $(l,l')\in J_{n,3}$. 
Instead of using $A''_{l',n}$ in (\ref{f6}) as we did when $d\ge 3$, 
we use more complicated events. 
Let
\begin{align*}
A'''_{l,l',n}:=A''_{l',n}\cap \{S_l\in \partial_b R[T_{S_{l'},l'}^{\lceil\beta n^2\rceil} ,u_{n}]\}.
\end{align*} 
Recall the definition of $A'_{l,n}$ in (\ref{f6}).  
By the definition of $A_{l,n}$ and ${A'}_{l,n}$, 
we have ${A}_{l,n}\subset A'_{l,n}$ and ${A}_{l,n}\cap A_{l',n}\subset A'''_{l,l',n}$. 
Note that $A'_{l,n}$ is not independent of $A'''_{l,l',n}$ for $(l,l')\in J_{n,3}$. 
Denote the event $0<|S_{l'}-S_l|<n \sqrt{u_{n-1}}$ by $D_1$, 
and the event $|S_{l'}-S_l|\ge n\sqrt{u_{n-1}}$ by $D_2$. 
Since $S_{l'}\notin R(l'-1)$ on $A_{l',n}$, 
we have $\{S_l=S_{l'}\} \cap A_{l',n}=\emptyset$ for $l<l'$. 
Thus, $A_{l',n}=(A_{l',n}\cap D_1)\cup(A_{l',n}\cap D_2)$ and therefore  
$A_{l,n} \cap A_{l',n}\subset 
(A'_{l,n} \cap A'''_{l,l',n}\cap D_1) 
\cup (\tilde{A}_{l,n} \cap \tilde{A}_{l',n} \cap D_2)$ holds. 
Then, the  following holds:
\begin{align}\label{se1}
E[1_{A_{l,n}}1_{A_{l',n}}]
\le
(E[1_{A'_{l,n}}1_{A'''_{l,l',n}}1_{D_1}]
+E[1_{\tilde{A}_{l,n}}1_{\tilde{A}_{l',n}}1_{D_2}]).
\end{align}
Hence, by putting (\ref{rw}) and (\ref{n1}) into the right hand side of the inequalities given in  Lemma \ref{subs} we can see that there exists $C>0$ 
such that for any $ l\in I_n$ and all sufficiently large $n\in \mathbb{N}$ with 
$u_{n-1}/(n-1)^{11}\le (\lceil u_{n-1}/n^{10}\rceil - k\lceil \beta n^2\rceil) \wedge l $
\begin{align}\notag
 &P(A'_{l,n})=P(F_{0,l,l+\lceil \frac{u_{n-1}}{n^{10}} \rceil,n})\\
\notag
\le &P(T_0 \wedge T_b >l )
\times P(T_0<T_b\wedge k)^{\lceil \beta n^2\rceil}
\times
P(T_b> \frac{u_{n-1}}{(n-1)^{11}})\\
\label{q2+}
\le& \frac{\pi }{2(n-1)^2}
\times P(T_0<T_b\wedge k)^{\lceil \beta n^2\rceil}
\times \frac{\pi }{(n-1)^2}
+C \frac{\exp((h-1)n^2)\log n}{n^6}.
\end{align}
Taking tha conditional probability of the event $A'_{l,n}\cap A'''_{l,l',n} \cap D_1$ on ${\cal F}(T_{S_{l'},l'}^{\lceil\beta n^2\rceil})$
and using  (\ref{q2+}), we see that for any  $ l,l' \in I_n$,
\begin{align}\notag
&E[1_{A'_{l,n}}1_{A'''_{l,l',n}}1_{D_1}]\\
\notag
=&E[1_{A'_{l,n}} 1_{D_1}  
1\{\{S_{l'}\notin R[l'-\lceil \frac{u_{n-1}}{n^{10}} \rceil,l'-1],\\
\notag
&S_{l'} \in \partial _b R[l'-\lceil \frac{u_{n-1}}{n^{10}} \rceil,T_{S_{l'},l'}^{\lceil\beta n^2\rceil}]\}\cap E_{l'n}\}\\
\label{eee}
&P(S_{l'},S_l\in\partial_b R[T_{S_{l'},l'}^{\lceil\beta n^2\rceil},u_{n}]
|{\cal F}(T_{S_{l'},l'}^{\lceil \beta n^2\rceil}))].
\end{align}
Note that $T_{x',l}^{\lceil\beta n^2\rceil} <u_{n-1}$ if $(l,l') \in J_{n,3}$. 
Hence, by (\ref{fo3}), we see that 
for all sufficiently large $n\in \mathbb{N}$ 
with $u_{n}-u_{n-1}\ge u_n/n$ and $x,x' \in {\mathbb{Z}^2}$ with $0<|x-x'|< n \sqrt{u_{n-1}}$, 
it holds that
\begin{align*}
&P(x,x'\in \partial_b R[T_{x',l}^{\lceil\beta n^2\rceil},u_{n}]
|{\cal F}(T_{x',l}^{\lceil \beta n^2\rceil }))\\
=&P(0, x-x' \in \partial_b R(u_{n}-t))|_{t=T_{x',l}^{\lceil\beta n^2\rceil}}\\
&\le \max_{0<|x-x'|< n \sqrt{u_{n-1}} }
P(T_b\wedge T_{x-x'+b}>\lceil  \frac{u_n}{n}\rceil ) 
\le \frac{\pi }{(n+1)^2}+\frac{C\log n}{n^4}.
\end{align*}
By the inequalities in the last line restricted to $x=S_l$ and $x'=S_{l'}$, 
the right hand side of (\ref{eee}) is bounded by 
\begin{align}\notag
&E[1_{A'_{l,n}}
1 \{ \{S_{l'}\in R[l'-\lceil \frac{u_{n-1}}{n^{10}} \rceil,l'-1]^c
\cap \partial_b R[l'-\lceil \frac{u_{n-1}}{n^{10}} \rceil,T_{S_{l'},l'}^{\lceil \beta n^2\rceil}\}\cap E_{l',n}\} ]\\
\label{cl}
&\times \bigg(\frac{\pi }{(n+1)^2}+\frac{C\log n}{n^4}\bigg).
\end{align}
Moreover, it holds that
\begin{align}\notag
&E[1_{A'_{l,n}}1\{ \{S_{l'}\in R[l'-\lceil \frac{u_{n-1}}{n^{10}} \rceil ,l'-1]^c
\cap \partial_b  R[l'-\lceil \frac{u_{n-1}}{n^{10}} \rceil,T_{S_{l'},l'}^{\lceil \beta n^2\rceil}]\}\cap E_{l',n}\}] \\
\notag
=&E[1_{A'_{l,n}}] E[1\{ \{S_{l'}\in R[l'-\lceil \frac{u_{n-1}}{n^{10}} \rceil ,l'-1]^c
\cap \partial_b R[l'-\lceil \frac{u_{n-1}}{n^{10}} \rceil,T_{S_{l'},l'}^{\lceil \beta n^2\rceil}]\}\cap E_{l',n}\}]\\
\label{cl*}
=&E[1_{A'_{l,n}}] 
E[1F(l'-\lceil \frac{u_{n-1}}{n^{10}} \rceil,l',T_{S_{l'},l'}^{\lceil \beta n^2\rceil},n )].
\end{align}
By substituting (\ref{ppp*}) for Remark \ref{pl}, we obtain 
\begin{align*}
&P(F(l'-\lceil \frac{u_{n-1}}{n^{10}} \rceil,l',T_{S_{l'},l'}^{\lceil \beta n^2\rceil},n ))\\
\le &\frac{\pi P(T_0<T_b \wedge k )^{\lceil\beta n^2\rceil}}{2(n-1)^2}
+C\frac{ \exp((h-1)n^2) \log n}{n^{4}}.
\end{align*}
Therefore, (\ref{q2+}), (\ref{cl}) and (\ref{cl*}) yield
\begin{align}\notag
&E[1_{A'_{l,n}}1_{A'''_{l,l',n}}1_{D_1}]\\
\label{smp1}
\le &\frac{\pi^4P(T_0<T_b \wedge k )^{2\lceil\beta n^2\rceil}}{4(n-1)^6(n+1)^2}
+C\frac{ \exp(2(h-1)n^2) \log n}{n^{10}}.
\end{align}

Now, we turn to the estimate of $E[1_{\tilde{A}_{l,n}}1_{\tilde{A}_{l',n}}1_{D_2}]$. 
From the large deviation result (see $(11)$ in \cite{Law3}), 
there exist $C$, $c>0$ such that 
for any $n$, $m\in\mathbb{N} \cap \{1\}^c$ with $m \le u_{n-1}$
\begin{align}\label{hhy}
P(|S_m|\ge n \sqrt{u_{n-1}})
\le Ce^{-cn}.
\end{align}
Thus, by the strong Markov property, we can estimate 
$E[1_{\tilde{A}_{l,n}}1_{\tilde{A}_{l',n}}1_{D_2}]$ 
for any $(l,l') \in J_{n,3}$ as
\begin{align}\notag
E[1_{\tilde{A}_{l,n}}1_{\tilde{A}_{l',n}}1_{D_2}]
= &E[1_{\tilde{A}_{l,n}}1_{D_2}]
E[1_{\tilde{A}_{l',n}}]\\
\notag
=&E[1_{\tilde{A}_{l,n}}
E[1_{D_2}|{\cal F}(T_{S_{l},l}^{\lceil \beta n^2\rceil })]]
E[1_{\tilde{A}_{l',n}}]\\
\notag
=&E[1_{\tilde{A}_{l,n}}
P(|S_{l'-l-t}|\ge n \sqrt{u_{n-1}})_{t=T_{S_{l}, l}^{\lceil \beta n^2\rceil} }]
E[1_{\tilde{A}_{l',n}}]\\
\notag
\le&E[1_{\tilde{A}_{l,n}}
\max_{|l-l'|\le u_{n-1}}P(|S_{l'-l}|\ge n \sqrt{u_{n-1}})]
E[1_{\tilde{A}_{l',n}}]\\
\notag
=&E[1_{\tilde{A}_{l,n}}]\max_{m\le u_{n-1}}
P(|S_{m}|\ge n \sqrt{u_{n-1}})
E[1_{\tilde{A}_{l',n}}]\\
\label{smp}
\le &C\frac{\exp(2(h-1)n^2)}{e^{cn}}.
\end{align}
The last inequality comes from (\ref{hj}) and (\ref{hhy}). 
Finally, by (\ref{q1}), (\ref{smp1}) and (\ref{smp}),  
we obtain the following estimate. 
Since $\sharp J_{n,3}\le (u_{n-1})^2$, for any $n\in {\mathbb N}\cap \{0\}^c$, 
\begin{align}\notag
&\sum_{(l,l')  \in J_{n,3}}
(E[1_{A'_{l,n}}1_{A'''_{l,l',n}}1_{D_1}]
+E[1_{\tilde{A}_{l,n}}1_{\tilde{A}_{l',n}}1_{D_2}]
-E[1_{A_{l,n}}]E[1_{A_{l',n}}])\\
\notag
\le &\sum_{(l,l') \in J_{n,3}}
 \bigg(\frac{\pi^4 P(T_0<T_b \wedge k)^{2\lceil\beta n^2\rceil}}{4(n-1)^6(n+1)^2}-
 \frac{\pi^4 P(T_0<T_b \wedge k )^{2\lceil\beta n^2\rceil}}{4(n-1)^4n^4}\\
\notag
 &+C\frac{\exp(2(h-1)n^2)}{e^{cn}}+C\frac{\exp(2(h-1)n^2)\log n}{n^{10}}
\bigg)\\
\notag
\le &C\sum_{(l,l')\in J_{n,3}}
\frac{\exp(2(h-1)n^2)\log n}{n^{10}}\\
\label{pp3+1}
\le&C\bigg(\frac{\exp(hn^2-2n)}{n^4}\bigg)^2\times \frac{\log n}{n^2}.
\end{align}
The second inequality comes from the fact that there exists $C>0$ such that for any $n\in \mathbb{N}\cap \{1\}^c$
\begin{align}\label{nnn}
\frac{1}{(n-1)^6(n+1)^2}-\frac{1}{(n-1)^4n^4}\le \frac{C}{n^{10}}. 
\end{align}
By (\ref{oo}), (\ref{o1}) and (\ref{pp3+1}),   
the right hand side of (\ref{formula2}) is bounded by 
$(\exp(hn^2-2n)/n^4)^2 \times \log n/n^2$. 
Therefore, we obtain the proof of Lemma \ref{hh} for $d=2$. 
\end{proof}

\begin{rem}\label{bb}
We observe what happens if we try to estimate the third term of the right hand side of 
(\ref{formula2}) 
in the case 
$d=2$ by the same argument as in the case $d\ge3$. 
Recall the definition of $A''_{l',n}$ in (\ref{f6}). 
Then, by substituting (\ref{ppp*}) for Lemma \ref{subs}, we can see that 
for any $ l'\in I_n$, 
$$P(A''_{l',n})\le \frac{\pi^2 P(T_0<T_b \wedge k )^{\lceil\beta n^2\rceil}}{2(n-1)^2n^2}+O(\frac{\exp((h-1)n^2)\log n}{n^{6}}).$$ 
Hence, if we choose $A''_{l',n}$ instead of $A'''_{l,l',n}$ in (\ref{se1}), 
\begin{align*}
E[1_{A'_{l,n}}1_{A''_{l',n}}]
=&E[1_{A'_{l,n}}]E[1_{A''_{l',n}}]\\
\le &\frac{\pi^4 P(T_0<T_b \wedge k)^{2\lceil\beta n^2\rceil}}{4(n-1)^6n^2}+O(\frac{\exp(2(h-1)n^2)\log n}{n^{10}}).
\end{align*}
Based on this estimate, we obtain  
\begin{align*}
&\sum_{(l,l')  \in J_{n,3}}
(E[1_{A'_{l,n}}1_{A''_{l',n}}]
-E[1_{A_{l,n}}]E[1_{A_{l',n}}])\\
\le &\sum_{(l,l') \in J_{n,3}}
 \bigg(\frac{\pi^4 P(T_0<T_b \wedge k )^{2\lceil\beta n^2\rceil}}{4(n-1)^6n^2}-
 \frac{\pi^4 P(T_0<T_b \wedge k)^{2\lceil\beta n^2\rceil}}{4(n-1)^4n^4}\\
&+O(\frac{\exp(2(h-1)n^2)\log n}{n^{10}})\bigg)\\
\le &C\sum_{(l,l') \in J_{n,3}}
\frac{\exp(2(h-1)n^2)}{n^9}\\
\le&C(\frac{\exp(hn^2-2n)}{n^4})^2\times \frac{1}{n}.
\end{align*}
The second inequality comes from the fact that there exists $C>0$ such that for any $n\in \mathbb{N}\cap \{1\}^c$
\begin{align*}
\frac{1}{(n-1)^6n^2}-\frac{1}{(n-1)^4n^4}\le \frac{C}{n^9}. 
\end{align*}
(Compare with (\ref{nnn}).) 
Consequently, with the aid of Lemma \ref{hh+}, we obtain
$$\frac{\mathrm{Var} (Q_n)}{c^2 (EQ_n)^2}\le \frac{C}{n}.$$
This estimate is not sufficient to apply the Borel-Cantelli lemma as we did in the proof of Theorem \ref{m1} in Section $4.1$. 
Thus, we cannot use the Borel-Cantelli lemma here.
\end{rem}

\section{Proof of Theorem \ref{m2}}
\subsection{Proof of the upper bound of Theorem \ref{m2} for $d\ge2$}
\begin{proof}
Note that if we substitute $\beta_d \delta$ for $\beta$ in Lemma \ref{theta},  we obtain $E[\tilde{\Theta}_n(\beta_d \delta)]=O(n^{1-\delta})$. 
By the Chebyshev inequality, we find that for any $\epsilon>0$, there exists $C>0$
\begin{align*}
P\bigg(\tilde{\Theta}_n(\beta_d\delta)\ge \bigg(\frac{n}{2} \bigg)^{1-\delta+\epsilon}\bigg)
<Cn^{-\epsilon}.
\end{align*}
Using the Borel-Cantelli lemma, we see that the events 
$\{\tilde{\Theta}_{2^k}(\beta_d\delta)\ge 2^{(k-1)(1-\delta+\epsilon)}\}$ 
happen only finitely often with probability one. 
Hence, it holds that for any $\epsilon>0$
\begin{align}\label{hhhj}
\limsup_{k\to \infty } \frac{\log \tilde{\Theta}_{2^k}(\beta_d\delta)}{\log 2^{k-1} }
\le 1-\delta+\epsilon \quad { a.s.}
\end{align}
Note that if $K(n,x)\ge \lceil \beta_d \delta \log n \rceil $, 
 for all $k$, $n\in {\mathbb N}$ with $2^{k-1}\le n<2^k$ 
$T_x^{\lceil \beta_d \delta \log 2^{k-1} \rceil } \le n$ and $K(2^k,x)\ge \lceil \beta_d \delta \log 2^{k-1} \rceil$,  
and hence $ {\Theta}_n(\delta) \le \tilde{\Theta}_{2^k}(\beta_d\delta)$  holds. 
Thus, for all $k$, $n\in {\mathbb N}$ with $2^{k-1}\le n<2^k$, we have
$$ \frac{\log  {\Theta}_n(\delta)}{\log n}
\le \frac{\log \tilde{\Theta}_{ 2^k }(\beta_d\delta)}{\log  2^{k-1} }.$$ 
Therefore, with (\ref{hhhj}) we obtain for any $\epsilon>0$, 
\begin{align*}
\limsup_{n\to \infty } \frac{\log  {\Theta}_n(\delta)}{\log n}
\le 1-\delta+\epsilon \quad { a.s.}
\end{align*}
The desired upper bound holds by combining these bounds.  
\end{proof}

\subsection{Proof of the lower bound of Theorem \ref{m2} for $d\ge2$}
\begin{proof}
We closely follow the argument in the proof of Lemma $4.2$ with $\beta=\beta_d \delta$. 
Take $k$ and $h_k$ as in section $4.1$. 
Set 
\begin{align*} 
W_n(\beta_d\delta)=\sharp \{ x\in \partial R(u_n) \cap R(u_{n-1}):
K(u_{n-1},x)\ge \lceil \beta_d\delta n^2\rceil \}.
\end{align*}
Note that $W_n(\beta_d\delta) \ge Q_n$ holds for any $n\in {\mathbb N}$. 
Indeed, if $x \in \partial_b R(u_n)$, then $x \in \partial R(u_n)$. 
Moreover, if $l\in I_n$, then $l+k\lceil\beta_d\delta n^2\rceil\le u_{n-1}$ holds for all sufficiently large $n\in {\mathbb N}$. 
Therefore, since we know (\ref{rrr}) and Lemma \ref{hh+}, we have
\begin{align*}
P(W_n(\beta_d\delta) \ge \frac{1}{2}EQ_n \ge \frac{c\exp(h_kn^2-2n)}{n^4}
 \quad\quad  \text{ for all but finitely many }n)=1.
\end{align*}
Let $u_{m-1}\le n < u_m$. 
Note that $K(u_{m-1,x})\le K(n,x)$ holds for all $x \in \mathbb{Z}^d$ 
and by virtue of (\ref{el*}), $\partial R(u_m) \cap R(u_{m-1}) \in \partial R(n)$. 
Hence, $W_m(\beta_d\delta) \le \Theta_n(\delta)$ holds. 
Therefore, it holds that
\begin{align*}
\liminf_{n \to \infty} \frac{\log \Theta_n(\delta)}{\log n}
\ge h_k
\quad \text{ a.s.}
\end{align*}
Since $h_k\to 1-(\beta_d\delta)/\beta_d$ as $k \to\infty$, the desired result holds, completing the proof. 
\end{proof}

\end{document}